\documentclass[11pt]{amsart}
\usepackage{amscd,amsfonts,mathrsfs,amssymb,amsmath}
\usepackage{hyperref}
\usepackage[margin=3.6 cm]{geometry}
\usepackage{verbatim}
\usepackage{xcolor}

\def \ch{{\bf H}_{\mathbb C}}

\def\V{\mathbb V}
\def\R{\mathbb R}

\def\H{\mathbb H}

\def\Z{\mathbb Z}

\def\C{\mathbb C}

\def\K{\mathbb K}

\def\P{\mathbb P}

\def\A{\mathbb A}

\def\PSp{\mathrm{PSp}}
\def\PU{\mathrm{PU}}

\def\Sp{\mathrm{Sp}}

\def\PSL{{\rm PSL}}

\def\z{{\bf z}}
\def\w{{\bf w}}

\newtheorem{thm}{Theorem}[section]
\newtheorem{lem}{Lemma}[section]
\newtheorem{rmk}{Remark}[section]

\newtheorem{cor}{Corollary}[section]

\newcommand{\thmref}[1]{Theorem~\ref{#1}}

\newcommand{\be}{\begin{equation}}
\newcommand{\ee}{\end{equation}}
\newcommand{\bea}{\begin{eqnarray}}
\newcommand{\eea}{\end{eqnarray}}
\newcommand{\Bea}{\begin{eqnarray*}}
\newcommand{\Eea}{\end{eqnarray*}}
\newcommand{\bt}{\begin{Theorem}}
\newcommand{\et}{\end{Theorem}}

\begin{document}

\title[On Free Group Generated by Two Heisenberg Translations]{On Free Group Generated by Two Heisenberg Translations} 
\author[Sagar B. Kalane]{Sagar B. Kalane}
\address{ Indian Institute of Science Education and Research (IISER) Pune,
Dr Homi Bhabha Rd, Ward No. 8, NCL Colony, Pashan, Pune, Maharashtra 411008}
\email{sagark327@gmail.com, sagarkalane@iiserpune.ac.in}
\author[D. Tiwari]{Devendra Tiwari}
\address{Bhaskaracharya Pratishthana, Erandwane, Damle Path, Off Law College Road, Pune Maharashtra 411004}
\email{devendra9.dev@gmail.com}
\date{\today}
\thanks{First author is supported by post doctoral institute fund of IISER Pune and second author is supported by ARSI-Foundation}
\begin{abstract}
In this paper, we will discuss the groups generated by two Heisenberg translations of ${\rm PSp}(2,1)$ and determine when they are free. The main result of the paper corrects an assertion made by Xie, Wang, Jiang in \cite{xwy} published in Canad. Math. Bull. $56(2013), 881-889,$ and from that we derive the result for the ${\rm PSp}(2,1)$ case.
\end{abstract}
\maketitle
\section{Introduction}

 To study discreteness of a group in ${\rm PSL}(2, \R)$ generated by prescribed elements is the central focus in the study of Fuchsian groups. Quasi-fuchsian theory expands this question to representations of surface
groups into ${\rm PSL}(2, \C)$ and similar low-dimensional Lie groups, and adds the question of faithfulness. In particular, Goldman and Parker \cite{gp} developed a complex hyperbolic quasi-fuchsian theory,
which started by examining representations of ideal triangle groups. This culminated in the proof of the Goldman-Parker conjecture by Schwartz \cite{RS} which states that a natural embedding of an ideal triangle group is discrete and faithful if and only if a certain ``angular invariant" is below a particular threshold. Work continues on the discreteness question for more intricate complex hyperbolic quasifuchsian groups and on the Higher Teichmuller Space of representations into other Lie groups.

\medskip It is an interesting problem to study when subgroup generated by two non-commuting parabolic elements is free. A similar problem in the case of Fuchsian and Kleinian groups studied by many authors. Especially in the works \cite{ak1, lyu, np}, authors have explored the conditions for two elements in Fuchsian or Kleinian groups to generate discrete free group. For example, Purzitsky \cite{np} found necessary and sufficient conditions for the subgroups generated by any pair $A, B \in \PSL(2, \R)$ to be discrete free product of cyclic groups $\langle A \rangle$ and $\langle B \rangle.$

\medskip A particular case of this classical problem is to find conditions when a subgroup $G$ of ${\rm PSL}(2, \C)$ generated by two non-commuting, parabolic, linear fractional transformations

$$A= \begin{pmatrix} 1 & \mu \\ 0 & 1 \end{pmatrix}, \; \;     B = \begin{pmatrix} 1 & 0 \\ \mu & 1 \end{pmatrix}$$

is free. 

\medskip This problem can be reformulated to find the set of complex numbers $\mu$ such that the corresponding group $G$ is free. Sanov and Brenner in their fundamental works \cite{lns, bre} provided the condition for $G$ to be free. Later Chang, Jennings, and Ree \cite{cjr} improved the previous condition to provide a weaker one. This was further improved by Lyndon and Ullman \cite{lu}, who gave very simple proof of results of Brenner \cite{bre} by using theorem of Macbeath.

\medskip In this paper, we focus on the work of Xie, Wang and Jiang \cite{xwy} in which they studied the groups generated by two non-commuting parabolic transformations of $\PU(2, 1)$ and derived conditions to determine when they are free. They consider two variants on the question and resolve them, for their Theorem $1.1$, by direct calculations and, for their Theorem $1.2$, by appealing to the ``faithful" part of
the Goldman-Parker-Schwartz Theorem. We find that the main result in \cite{xwy} i.e. {\rm Theorem $1.2$} of of Xie, Wang and Jiang doesn't work for all $\mu$ as suggested by theorem. 

\medskip In this work we correct their work and improved the condition for the resulting group to be free. As by product we obtained similar results for groups generated by two Heisenberg translation of $\PSp(2, 1)$. We further provide, using a classical result of Lyndon and Ullman \cite{lyu}, a very simple proof of {\rm Theorem $1.1$} of \cite{xwy}. Now we will state main result of our paper which is improved version of the result proved in \cite{xwy}(Theorem $1.2$).
\begin{thm} \label{gh}
	
	Let $\mu \in \C$ and assume that the subgroup ${\rm G} \subset \PU(2,1)$ is generated by
	
	$$ A= \begin{pmatrix} 1 & 2\sqrt{2} & -4 \\
	0 & 1 & -2\sqrt{2} \\ 0 & 0 & 1 \end{pmatrix} \; \text{and} ~\; B = \begin{pmatrix} 1 & 0 & 0 \\ 2\sqrt{2}\mu & 1 & 0 \\ -4|\mu|^2 & -2\sqrt{2}\bar{\mu} & 1
	\end{pmatrix}.$$
	If 
	$$  |\mu|^2 = -\Re(\mu) \geq \frac{3}{128}$$
	
	then ${\rm G}$ is freely generated by $A$ and $B.$
\end{thm}

\begin{rmk} 
The reason for the error in the {\rm Theorem 1.2} of \cite{xwy} is that, in order to apply the the result of Schwartz \cite{RS}, one needs $AB$ to be unipotent and they made a mistake in the calculation. To be precise, instead of computing {\rm arg}(z), they computed $\Im(z)$ (for details see Section $3$).
\end{rmk}

\medskip As an immediate corollary, we have the following result in case of $\PSp(2, 1)$.

\begin{cor} \label{ghc}Let $\mu \in \H$ and assume that the subgroup ${\rm G} \subset {\rm PSp}(2,1)$ is generated by
	$$ A= \begin{pmatrix} 1 & 2\sqrt{2} & -4 \\
	0 & 1 & -2\sqrt{2} \\ 0 & 0 & 1 \end{pmatrix} \; \text{and} ~\; B = \begin{pmatrix} 1 & 0 & 0 \\ 2\sqrt{2}\mu & 1 & 0 \\ -4|\mu|^2 & -2\sqrt{2}\bar{\mu} & 1
	\end{pmatrix}.$$
	If 
	$$  |\mu|^2 = -\Re(\mu) \geq \frac{3}{128},$$
	then ${\rm G}$ is freely generated by $A$ and $B.$

\end{cor}
After discussing some background material in Section $2$, we prove Theorem $1.2$ and Corollary $1.2$ in the Section $3.$

\section{Preliminaries}\label{prel}

\subsection{The Quaternions}Let $\H$ denote the division ring of quaternions.  Any $ q\in {\H}$ can be uniquely written as  $q=r_{0}+r_{1}i+r_{2}j+r_{3}k$, where $r_{0},r_{1},r_{2},r_{3}\in \R$, and  $ i,j,k$ satisfy relations:  $i^{2}=j^{2}=k^{2}=ijk=-1$. The real number $r_0$ is called the real part of $q$ and we denote it by $\Re(q)$. The imaginary part of $q$ is $\Im(q)=r_{1}i+r_{2}j+r_{3}k$. The conjugate of $q$ is defined as $\overline {q}= r_0-r_1 i -r_2 j - r_3 k$.
The norm of $q$ is $|q|=\sqrt{r_0^{2}+r_1^{2}+r_2^{2}+r_3^{2}}$.

\subsection{Hyperbolic Space}
Let $\K = \C$ or $\H$ and suppose $\V=\K^{2,1}$ be the $3$ dimensional right vector over $\K$ equipped with the Hermitian form of signature $(2,1)$ given by 
$$\langle\z,\w\rangle=\w^{\ast}H\z=\bar w_3 z_1+\bar w_2 z_2+\bar w_1 z_3,$$
where $\ast$ denotes conjugate transpose. The matrix of the Hermitian form is given by
\begin{center}
$H=\left[ \begin{array}{cccc}
            0 & 0 & 1\\
           0 & 1 & 0 \\
 1 & 0 & 0\\
          \end{array}\right].$
\end{center}
\medskip We consider the following subspaces of $\K^{2,1}:$
$$\V_{-}=\{\z\in\K^{2,1}:\langle\z,\z \rangle<0\}, ~ \V_+=\{\z\in\K^{2,1}:\langle\z,\z \rangle>0\},$$
$$\V_{0}=\{\z-\{{\bf 0}\}\in\K^{2,1}:\langle\z,\z \rangle=0\}.$$
\medskip A vector $\z$ in $\K^{2,1}$ is called \emph{positive, negative}  or \emph{null}  depending on whether $\z$ belongs to $\V_+$,   $\V_-$ or  $\V_0$. There are two distinguished points in $\V_{0}$ which we denote by  $\bf{o}$ and $\bf\infty,$ given by
$$\bf{o}=\left[\begin{array}{c}
               0\\0\\  1\\
              \end{array}\right],
~~ \bf\infty=\left[\begin{array}{c}
               1\\0 \\ 0\\
              \end{array}\right].$$

\medskip Let $\P:\K^{2,1}-\{0\}\longrightarrow  \K \P^2$ be the right projection onto the respective projective space. Image of a vector $\z$ will be denoted by $z$.  The respective $\K$-hyperbolic space $\bf{H}^2_{\mathbb{K}}$ is defined to be $\P \V_{-}$. The ideal boundary of $\bf{H}^2_{\mathbb{K}}$ is $\partial \bf{H}^2_{\mathbb{K}}=\P \V_{0}$. It is clear that $\infty\in\partial \bf{H}^2_{\mathbb{K}}$. We mention that  $g\in {\rm U}(2, 1)$ acts on
$\bf{H}^2_{\mathbb{K}} \cup \partial \bf{H}^2_{\mathbb{K}}$ as $
g(z)=\P g\P^{-1}(z).$ For a point $\z=\begin{bmatrix}z_1 & z_2 &  z_{3}\end{bmatrix}^T \in \V_- \cup \V_0$, projective coordinates are given by $(w_1, w_2)$, where $w_i=z_i z_{3}^{-1}$ for $i=1,2$. In projective coordinates we have 

$$\bf{H}^2_{\mathbb{K}} =\{(w_1,w_2)\in \K^2 : \ 2\Re(w_1)+|w_2|^2<0\},$$
$$\partial \bf{H}^2_{\mathbb{K}}=\P \V_{0}-\infty=\{(w_1,w_2)\in \K^2 :2\Re(w_1)+|w_2|^2=0\}.$$

\medskip This is the Siegel domain model of $\bf{H}^2_{\mathbb{K}}$. Similarly one can define the ball model by replacing $H$ with a equivalent Hermitian form given by the diagonal matrix having one diagonal entry $-1$ and all other diagonal entries $1$. We shall mostly use the Siegel domain model here. For more details we refer to \cite{chen}.

\medskip Given a point $z$ of $\bf{H}^2_{\mathbb{K}} \cup \partial \bf{H}^2_{\mathbb{K}} -\{\infty\} \subset\K \P^2$ we may lift $z=(z_1,z_2)$ to a point $\z$ in $\V_0 \cup \V_{-}$, called the \emph{standard lift} of $z$. It is represented in non-homogeneous coordinates by
               
 $$ {\bf z} = \begin{pmatrix} z_1 \\ z_2 \\ 1 \end{pmatrix}, \; \; z_1 +\bar{z_1} + |z_2|^2 = 0. $$

 A finite point $z$ lies in the boundary of the Siegel domain if its standard lift to $\K^{2, 1}.$

\medskip We write $\zeta = \frac{z_2}{\sqrt{2}} \in \K,$ and this condition becomes $2\Re(z_1) = -2|\zeta|^2.$ Hence we may write $z_1 = - |\zeta|^2 + \nu$ for $\nu \in \Im(\K).$ Therefore for $\zeta \in \K$ and $\nu \in  \Im(\K)$,

$$ {\bf z} = \begin{pmatrix} -|\zeta|^2 + \nu \\ \sqrt{2}\zeta \\ 1 \end{pmatrix}.$$

\medskip The finite points in the
boundary of $\bf{H}^2_{\mathbb{K}}$  naturally carry the structure of the
generalised Heisenberg group $\mathfrak{N}_2$, which is defined to
be $\mathfrak{N}_2 = {\K}\times \Im({\K})$, with the group
law
\begin{equation}\label{heilaw}(\zeta_1, \nu_1)(\zeta_2,\nu_2)=(\zeta_1+\zeta_2, \nu_1+\nu_2+2\Im(\zeta_2^*\zeta_1)).
\end{equation}

\subsection{Heisenberg Translation}
The Heisenberg group acts on itself by Heisenberg translation. For $(\zeta_0,
\nu_0) \in \mathfrak{N}_2,$ this is given by

$$ T(\zeta_0, \nu_0) : (\zeta, \nu) \mapsto (\zeta_0 + \zeta, \nu+\nu_0 + 2\Im(\zeta^{*}\zeta_0) = ((\zeta_0,
\nu_0)(\zeta,
\nu))$$

\medskip Heisenberg translation is one type of parabolic element fixing $\infty$, as a matrix $T_{(\zeta_0, \nu_0)}$ is given by
\begin{equation}\label{htrans} T_{(\zeta_0, \nu_0)}=\left(
     \begin{array}{ccc}
       1 & -\sqrt{2}\zeta_0^{*} & -|\zeta_0|^2+\nu \\
       0 & 1 & \sqrt{2}\zeta_0 \\
       0 & 0 & 1 \\
     \end{array}
   \right).
\end{equation}
\medskip When $\zeta_0=0$, $T_{(\zeta_0, \nu_0)}$ is called the vertical translation;
Otherwise, $T_{(\zeta_0, \nu_0)}$ is called the non-vertical translation.

\section{Main Result}

\medskip In this section we improve the results of \cite{xwy} for general Heisenberg translations in $\PU(2, 1)$. If we take $\mu = \frac{-3}{4}$, then it satisfies the following condition of {\rm Theorem 1.2} in \cite{xwy}:

$$\frac{1}{\sqrt{1 +\left(\tan^{-1}\left(\sqrt{\frac{125}{3}}\right)\right)^2}} \leq |\mu|= \frac{3}{4} .$$
On the other hand, for $\mu = \frac{-3}{4}$ we get
$$AB= \begin{pmatrix}
4 & -4\sqrt{2} & -4 \\ 3\sqrt{2} & -5 & -2\sqrt{2} \\ \frac{-9}{4} & \frac{3}{\sqrt{2}} & 1 \end{pmatrix}.$$

Here ${\rm tr}(AB) = 0$ and thus $AB$ has order $3,$ which implies that the subgroup generated by $A$ and $B$ is not free.


\medskip The reason for this error is the following. In order to apply the result of Schwartz \cite{RS}, one needs $AB$ to be unipotent. Observe that, ${\rm tr}(AB) = 3 + 16\Re(\mu) + 16|\mu|^2$ and hence we need, $$\Re(\mu) + |\mu|^2 = 0.$$

\medskip If we take two vertical Heisenberg translations $A$ and $B$ which fixes $\infty$ and $ o,$ respectively, then we will get a simple proof of Theorem $1.1$ in \cite{xwy} by using the following result in \cite{lu}. 
\begin{lem}\cite{lu} \label{lu1}
	If $m,n \in \C$ and $|mn| \geq 4$, then 
	
	$$A_o= \begin{pmatrix}
1 & m \\ 0 & 1 \end{pmatrix} \;\text{and} ~ \; B_o = \begin{pmatrix}
1 & 0 \\  n & 0 \end{pmatrix},$$ freely generate a free group.
\end{lem}

As proof of following theorem is very simple, so we include it here for completeness.

\begin{thm}\label{sd}
	Suppose the subgroup $G \subset {\rm PU}(2,1)$ is generated by 
	
	$$A_1= \begin{pmatrix}
	1 & 0 & m \\ 0 & 1 & 0 \\ 0 & 0 & 1 \end{pmatrix} \;\text{and} ~ \; B_1 = \begin{pmatrix}
	1 & 0 & 0 \\ 0 & 1 & 0 \\ n & 0 & 1 \end{pmatrix}.$$
	
	If $|mn|\geq 4$, then $G$ is free.
\end{thm}
\begin{proof}
	Consider the following embedding $\phi$ of $GL(2,\C)$ in $GL(3,\C)$ defined by $$\phi\left( \begin{pmatrix}
	a & b \\ c & d \end{pmatrix}\right)= \begin{pmatrix}
	a & 0 & b \\ 0 & 1 & 0 \\ c & 0 & d \end{pmatrix}.$$ Thus under this embedding, image of $\langle A_0,B_0\rangle$ is $G=\langle A_1,B_1\rangle.$ Hence by using Lemma \ref{lu1}, $G$ is free if $|mn|\geq 4.$
\end{proof}
We obtain the following corollary of Theorem\ref{sd} for the vertical Heisenberg translations in ${\rm PSp}(2,1)$.

\begin{cor} \label{vh}
	Suppose the subgroup $G \subset {\rm PSp}(2,1)$ is generated by 
	$$A = \begin{pmatrix}
	1 & 0 & \tau \\ 0 & 1 & 0 \\ 0 & 0 & 1 \end{pmatrix} \; \text{and}~ \; B = \begin{pmatrix}
	1 & 0 & 0 \\ 0 & 1 & 0 \\ \tau & 0 & 1 \end{pmatrix},$$
	where $\tau$ is purely imaginary quaternion, i.e., $ \tau = t_1i + t_2j +t_3k.$  If $|\tau|\geq 2$, then $G$ is free.
\end{cor}
\begin{proof}
	$A$ and $B$ can be simultaneously conjugated to $A_1$ and $B_1$, respectively by the conjugation element
	$C={\rm diag}(\alpha, 1, \alpha),$ where $\alpha \tau \alpha^{-1}=ti$. Here $\alpha$ is a unit quaternion and $t=|\tau|=\sqrt{t_1^2 + t_2^2 + t_3^2}$.  Now the result follows from \thmref{sd}.
\end{proof}

Before proving the main theorem, we state some lemmas used in our proof of the theorem. By a classical theorem of Kurosh, the subgroup of a free product group have a particularly simple structure as described in the following lemma:

\begin{lem}{(Kurosh Subgroup Theorem)} Let $G = \ast_{i\in I}F_i$ be the free product of a collection of group $F_i$. If $H$ is a subgroup of $G$, then $H$ decomposes as a free product of the form

$$ H = F \ast(\ast_{i \in I}(\ast_{j \in J(i)} A \cap u_jF_iu_j^{-1})),$$
where $F$ is a free group. Consequently, $H$ is the free product of a free group and of various subgroups that are the intersections of $H$ with conjugates of the $F_i$.
\end{lem}

\medskip Applying the Kurosh subgroup theorem to the free product of infinite cyclic groups gives the following lemma, which is proved in \cite{xwy}.

\begin{lem} \label{lem1}
Let $G = \langle a \rangle \ast \langle b \rangle \ast \langle c \rangle$ be the free product of the three cyclic subgroups of order two  and $H = \langle ab, bc \rangle$,  then $H$ is free.
\end{lem}

\medskip Let $(p_0, p_1, p_2)$ be a triple of distinct points in $\partial \ch^2$. The points $p_0, p_1, p_2$ determine complex geodesics $C_0, C_1, C_2$ that are fixed by the inversions $i_0, i_1, i_2$ respectively. For $p=(p_0, p_1, p_2),$ Cartan angular invariant is defined by
$$\A(p)=\A((p_0, p_1, p_2)) = {\rm arg} \left( - \langle {\bf p_0}, \bf{p_1} \rangle \langle \bf{p_1}, \bf{p_2} \rangle \langle \bf{p_2}, \bf{p_0} \rangle \right) \in [-\frac{\pi}{2}, \frac{\pi}{2}],$$

where $\bf{p_0}, \bf{p_1}, \bf{p_2}$ are lifts of  $p_0, p_1, p_2,$ respectively.

\medskip
Define,

$$\rho_{\mathbb{A}} : \Z/2\ast \Z/2\ast \Z/2 \rightarrow {\rm PU}(2,1),$$

$$ (a, b, c) \mapsto (i_0, i_1, i_2).$$
 
 \medskip In this regard, we have the following result by Goldman-Parker-Schwartz, in\cite{RS}. 

\begin{lem}\label{lem2} {\rm (Goldman-Parker-Schwartz Theorem)} Let $i_0, i_1, i_2$ be the inversions of complex geodesics defined by above. Then subgroup generated by $i_0, i_1, i_2$ is discrete if and only if $|\A| \leq \tan^{-1}(\sqrt{\frac{125}{3}}).$ Furthermore, the group $\langle i_0, i_1, i_2 \rangle$ freely generates the free product  $\langle i_0, \rangle \ast \langle i_1 \rangle \ast \langle i_2 \rangle$ of the three cyclic groups of order 2.
\end{lem}

\medskip \subsection{Proof of the Theorem \ref{gh}}

\begin{proof}
Let $A$ and $B$ be the given generators having fixed points $p_0 = 0$ and $p_1 = \infty$, respectively. Let $p_2 = ( \zeta, \nu )$ be an arbitrary point from $\partial \ch^2$, which is distinct from $p_0$ and $p_1$. 
Consider $p = (p_0, p_1, p_2)$ be a triple of points in $(\partial \ch^2)^3$ and let $C_0$ be the complex geodesic  $\overrightarrow{p_1p_2}$ spanned by $p_1$ and $p_2$.
Similarly, let $C_1 =\overrightarrow{p_0p_2}$ and $C_2 = \overrightarrow{p_0p_1}$. Here the complex geodesics $C_j$ are given by
$$C_j = \P ( \lbrace z \in \C^{2, 1} | \langle z , c_j \rangle = 0 \rbrace ),$$ for $j = 0, 1, 2.$ 

Define the coordinates of lifts of $p_0, p_1, p_2$ into $\C^{2, 1}$ as:

$$\bf{p_0} = \begin{pmatrix} 0 \\ 0 \\ 1 \end{pmatrix}, \; \;  \bf{p_1} = \begin{pmatrix} 1 \\ 0 \\ 0 \end{pmatrix}, \; \;  \bf{p_2} =\begin{pmatrix} -|\zeta|^2 + vi \\ \sqrt{2}\zeta \\ 1 \end{pmatrix}.$$

Also, we can define the polar vectors corresponding to the above geodesics as follows:
$$c_0 = \begin{pmatrix} \sqrt{2}\bar{\zeta} \\ -1 \\ 0 \end{pmatrix}, \; \; c_1 = \begin{pmatrix} 0 \\ \frac{-|\zeta|^2 - i\nu}{\sqrt{|\zeta|^4 + |\nu|^2}} \\ \frac{-\sqrt{2}\bar{\zeta}}{\sqrt{|\zeta|^4 + |\nu|^2}} \end{pmatrix}, \; \; c_2 = \begin{pmatrix} 0 \\ 1 \\ 0 \end{pmatrix}.$$

\medskip Thus the inversions around $C_j$ are given by

$$i_j(Z) = -Z + \frac{2\langle Z, c_j\rangle}{\langle c_j, c_j\rangle}c_j, \quad j=0, 1, 2.$$

We can write the inversion generators in $\PU(2, 1)$ as:

$$ i_0 = \begin{pmatrix} -1 & -2\sqrt{2}\bar{\zeta} & 4|\zeta|^2 \\
       0 & 1 & -2\sqrt{2}\zeta \\ 0 & 0 & -1 \end{pmatrix} ,\quad  i_1 = \begin{pmatrix} -1 & 0 & 0 \\
       \frac{2\sqrt{2}\zeta(|\zeta|^2 + i\nu)}{|\zeta|^4 + \nu^2} & 1 & 0 \\ \frac{4|\zeta|^2}{|\zeta|^4 + \nu^2} & \frac{-2\sqrt{2}\bar{\zeta}(-|\zeta|^2 + i\nu)}{|\zeta|^4 + \nu^2} & -1 \end{pmatrix}$$

       and 
       $$ i_2 = \begin{pmatrix} -1 & 0 & 0 \\
       0 & 1 & 0 \\ 0 & 0 & -1 \end{pmatrix}.$$
   
Now we will determine the coordinates of $p_2$ such that $$ A = i_0i_2, \; \;\quad B = i_2i_1.$$

\medskip By comparing the entries of $i_0 i_2$  with $A$ we get $$ 2\sqrt{2}{\zeta} = -2\sqrt{2}.$$ 
Therefore, we have $\zeta = -1.$

\medskip Similarly, after comparing the entries of $i_2i_1$ with $B$ we obtained
$$\frac{-4|\zeta|^2}{|\zeta|^4 + \nu^2} = -4|\mu|^2 \quad,\quad  2\sqrt{2} \mu = \frac{2\sqrt{2}\zeta(|\zeta|^2 + i\nu)}{|\zeta|^4 + \nu^2}.$$
Substituting $\zeta = -1$ above we have
$$\mu = \frac{-1 -i\nu}{1 + \nu^2}, \quad |\mu|^2 = \frac{1}{1 + \nu^2}.$$
Hence $\mu$ satisfies the following condition
\begin{equation}
\Re(\mu) + |\mu|^2 = 0.\end{equation}
Now, if $\mu= a+bi$ then we get $a^2+b^2+a=0$. That is, $a^2+a+1/4+b^2=1/4.$
This restricts $\mu$ to lie on the circle $|\mu+{\frac{1}{2}}|={\frac{1}{2}}.$

\medskip Now, we will calculate the angular invariant $\A(p_0, p_1, p_2)$,
Since,  
       $$ {\langle \bf{p_0} , \bf{p_1}  \rangle} = 1, \quad {\langle \bf{p_1}, \bf{p_2} \rangle} = 1 \quad \text{and} \quad {\langle \bf{p_2}, \bf{p_0} \rangle} = -1 +i\nu$$
       
       we have, 
       
       \begin{equation}
       \A(p_0, p_1, p_2) = {\rm arg} (1 - i\nu).
\end{equation}      

Observe that
$$\tan(\A({ p})) = \frac{\Im(1 - i\nu)}{\Re(1 - i\nu)} = -\nu.$$

\medskip Thus by Lemma \ref{lem2}, we see that the group $\langle i_0, i_1, i_2 \rangle$ freely generates the free product $\langle i_0 \rangle \ast \langle i_1 \rangle \ast \langle i_2 \rangle$  when 

$$|\tan \A(p) | \leq \sqrt{\frac{125}{3}}.$$
So we get that,
 $$ |-\nu| = |\nu| \leq \sqrt{\frac{125}{3}}.$$
By applying Lemma \ref{lem1}, we see that $\langle i_0i_2, i_1i_2 \rangle = \langle A, B \rangle$ is free if  $$|\nu|^2 \leq \frac{125}{3}.$$

Since, $|\mu|^2 = \frac{1}{1 + |\nu|^2},$ we have, $|\nu|^2 = -1 + \frac{1}{|\mu|^2},$ i.e.,  $  |\mu|^2 \geq \frac{3}{128}.$
       
    Now by using equation $(3)$, we have  $$  |\mu|^2 = -\Re(\mu) \geq \frac{3}{128}.$$

\end{proof}
\medskip \subsection{Proof of the Corollary \ref{ghc}}
\begin{proof}

Let 
$$A_0 = A \quad \text{and}~ \quad  B_0 = \begin{pmatrix} 1 & 0 & 0 \\ 2\sqrt{2}\tau & 1 & 0 \\ -4|\tau|^2 & -2\sqrt{2}\bar{\tau} & 1
\end{pmatrix},$$

where $$\tau = \Re(\mu) + \sqrt{|\mu|^2 -  \Re(\mu)^2}~i.$$
By conjugating $A_0$ and $B_0$ by diagonal matrix $D={\rm diag}(\mu, \mu, \mu)$, where $\mu \in \Sp(1)$, we get $A= DA_0 D^{-1}$ and $B= DB_0 D^{-1},$ respectively. 
Now applying Theorem \ref{gh} on $\langle A_0, B_0 \rangle$, we see that $\langle A_0, B_0 \rangle$ is free if

\[\frac{3}{128} \leq|\tau|^2= -\Re(\tau) \;.\] 

Consequently, if 
 $  |\mu|^2 = -\Re(\mu) \geq \frac{3}{128}$
then ${\rm G}$ is freely generated by $A$ and $B.$
\end{proof}

\end{document}